\documentclass[letterpaper, 10 pt, conference]{ieeeconf}  

\IEEEoverridecommandlockouts                              

\usepackage{graphics} 
\usepackage{amsmath} 
\usepackage{amssymb}  
\usepackage{makecell} 

\usepackage[caption=false, font=footnotesize]{subfig}

\usepackage{multiaudience}
\SetNewAudience{ral}
\SetNewAudience{arxiv}

\usepackage{mathtools} 
\usepackage{todonotes}
\makeatletter
\def\endthebibliography{%
	\def\@noitemerr{\@latex@warning{Empty `thebibliography' environment}}%
	\endlist
}
\makeatother
\title{\LARGE \bf
	Analysis of minima for geodesic and chordal cost for a minimal 2D pose-graph SLAM problem
}

\author{Felix H. Kong$^{1}$, Jiaheng Zhao$^{1}$, Liang Zhao$^{1}$, and Shoudong Huang$^{1}$
	\thanks{$^{1}$The authors are with the Centre for Autonomous Systems, Faculty of Engineering and Information Technology, University of Technology Sydney (UTS), Sydney, Australia. Email:
		{\tt\small \{felix.kong, jiaheng.zhao, liang.zhao,shoudong.huang\}@uts.edu.au}}%
}

\newcommand{\R}{\ensuremath{\mathbb R}}

\newcommand{\inv}{\ensuremath{^{-1}}}
\newcommand{\wrap}{\ensuremath{\mathrm{wrap}}}
\newcommand{\diag}{\ensuremath{\mathrm{diag}}}
\newcommand{\atanTwo}{\ensuremath{\mathrm{atan2}}}

\newcommand{\transp}{\ensuremath{^{^\intercal}}}

\newtheorem{proposition}{Proposition}
\newtheorem{theorem}{Theorem}
\newtheorem{remark}{Remark}
\newtheorem{lemma}{Lemma}

\begin{document}

	\maketitle
	\thispagestyle{empty}
	\pagestyle{empty}

	\begin{abstract} 
		In this paper, we show that for a minimal 2D pose-graph SLAM problem, even in the ideal case of perfect measurements and spherical covariance, using \textit{geodesic distance} (in 2D, the ``wrap function'') to compare angles results in multiple suboptimal local minima.  We numerically estimate regions of attraction to these local minima for some examples, give evidence to show that they are of nonzero measure, and that these regions grow in size as noise is added. In contrast, under the same assumptions, we show that the \textit{chordal distance} representation of angle error has a unique minimum up to periodicity. For chordal cost, we find that initial conditions failing to converge to the global minimum are far fewer, fail because of numerical issues, and do not seem to grow with noise in our examples. 
		
		\textit{Keywords} --- Pose-graph SLAM, convergence analysis
	\end{abstract}

	\section{INTRODUCTION}
	Simultaneous Localization and Mapping (SLAM) is concerned with simultaneously estimating the pose of a robot (localization) and building a map of its surroundings (mapping). This capability has been useful in many areas, such as unmanned aerial vehicles \cite{Weiss2011MonocularSLAMBased,Kim2007RealTimeImplementation}, autonomous ground vehicles \cite{Geiger2012AreWeReady}, \cite{Lategahn2011VisualSLAMAutonomous}, and a plethora of other applications \cite{Cadena2016PastPresentFuture}. 
	
	Currently, the ``modern'' approach to SLAM is to represent the robot's trajectory as a graph: that is, to represent the robot's poses as nodes, and measurements from those poses as edges. Then, given this graph, typically a weighted least-squares optimization problem is solved to estimate the most likely robot poses given the robot's measurements \cite{Cadena2016PastPresentFuture}. However, even in 2D SLAM, the optimization problem is usually nonlinear and nonconvex, resulting in the possibility for iterative solvers to converge to a local instead of global minimum. 
	Although modern solvers appear to achieve a global minimum much of the time, it is as of yet unclear under what conditions local minima exist, and how many there are, even for very small problems.

	It is known that the cost on error in orientations of poses is a major contributor to the nonlinearity of the problem \cite{Huang2010HowFarIs,Wang2013StructureNonlinearitiesPose}. Because of this, choice of a particular representation of orientation error can affect the existence, number, and nature of minima in a pose-graph optimization problem.
	In 2D SLAM, one common method to evaluate orientation error is to directly subtract the two (scalar) angles, then ``wrap'' this difference to be on the interval $[-\pi,\pi)$, resulting in the ``geodesic distance'' between two orientations. Open-source SLAM software implementations such as Google's Cartographer \cite{Hess2016RealTimeLoop}, and also other popular software such as MATLAB's Navigation Toolbox implementation use geodesic distance via the ``wrap'' function in their cost functions. Representing orientation error using geodesic distance is intuitive and widely known; however, it has been linked empirically to convergence to suboptimal local minima \cite{Wang2018Comparisontwodifferent}. 
	
	Another method to evaluate orientation error is to use the ``chordal distance'' (e.g. \cite{Carlone2016PlanarPoseGraph,Carlone2015LagrangianDuality3D}), which is calculated using the Frobenius norm of the difference in rotation matrices. 
	The use of the chordal distance in the cost function of a pose-graph optimization problem has been investigated in several previous works. 
By using chordal cost and reformulating the problem as an equality-constrained optimization problem, global optimality can be certified for a pose-graph problem by solving a semidefinite program  \cite{Carlone2015LagrangianDuality3D,Carlone2015DualitybasedVerification}. Built upon this work, methods to speed up the computation of the global optimality certificate are proposed in \cite{Briales0609FastGlobalOptimality}. Solvers have also been developed that yield certifiably globally optimal solutions \cite{Rosen2019SESync,Briales2017CartnSyncFast,Mangelson2019GuaranteedGloballyOptimal}.
	
			\begin{figure}
		\centering
		
		\subfloat[Regions $\mathcal R_k$, marked by red lines. \label{fig:regions}]{%
			\includegraphics[height=0.2\textheight]{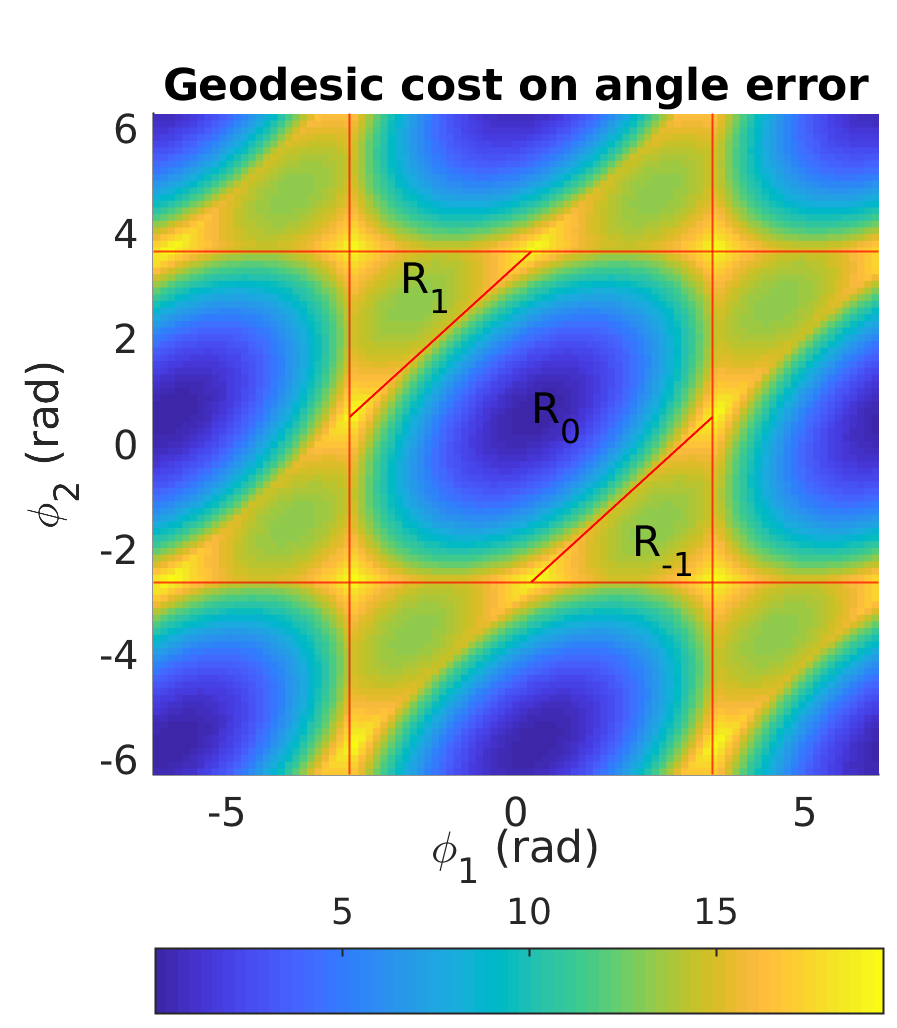}}
		\hfill
		\subfloat[Poses at local minimum using geodesic cost. \label{fig:localminimumslamsolution}]{%
			\includegraphics[height=0.2\textheight]{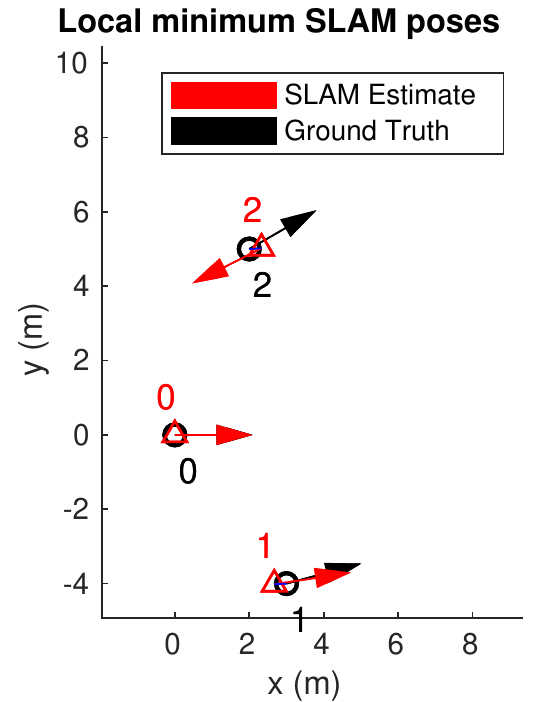}}
		
		\caption{(a) This heatmap motivates why there may be local minima when using geodesic distance. This is a plot of one period of $F_\phi(\Phi)$, the angular part of the cost function using geodesic distance. Clearly, local minima exist in $\mathcal R_1$ and $\mathcal R_{-1}$, which suggest the that there may be local minima in the full cost function. See Section \ref{sec:angularCost} for details. (b) Poses corresponding to a local minimum of a minimal 2D pose-graph problem using geodesic cost. Note the large orientation error of pose 2. See Section \ref{sec:exampleAndDiscussion} for more detail about this example problem. }
		\label{fig:group1}
	\end{figure}
	
	While practical SLAM problems are much larger, analyzing a small problem allows clear conclusions to be made, which can inform insights into larger problems. Several papers have investigated ``minimal'' SLAM problems in an attempt to show the fundamental structure and limitations of different formulations of SLAM. In the formulation in \cite{Wang2013StructureNonlinearitiesPose}, the authors concluded that for noise in some bounded interval, there is a unique global minimum, and no local minima. However, the authors assumed that the angle differences are always within $[-\pi,\pi)$. This allowed the angular terms to be treated as linear, which facilitates analysis, but cannot be assumed in general. In another paper \cite{Wang2018Comparisontwodifferent}, the authors compare the use of geodesic and chordal distance in a feature-based SLAM problem with two robot poses and a single landmark, which is considered to have no orientation. For that problem, using chordal cost, the authors concluded that a unique global minimum exists, regardless of noise.
	
	This paper compares the influence of using geodesic or chordal distance on the convergence properties of a minimal pose-graph problem. In this paper, we study a planar pose-graph problem with three poses and three measurements. Our contributions are:
	\begin{itemize}
		\item We prove that even in the case of perfect measurements with spherical covariance, if geodesic distance is used, multiple suboptimal local minima exist (see Figure \ref{fig:localminimumslamsolution} for an example of a local minimum). This is a direct result of representing orientation error using geodesic distance (see Figure \ref{fig:regions}). This clarifies the work in \cite{Wang2013StructureNonlinearitiesPose}; in particular, answers the question of what happens when angle differences outside of $[-\pi,\pi)$ are considered.
		\item We numerically estimate the regions of attraction to the local minima for some examples with varying noise magnitude, and show that they are of nonzero measure. Conservative regions of attraction to global minimum have been investigated for higher-dimensional problems using Gauss-Newton \cite{Carlone2013ConvergenceAnalysisPose}; in this paper, due to the small size of the problem, we can approximate the size of the region of attraction to the global (and local) minima, with little conservativism. 
		\item  We build upon \cite{Wang2018Comparisontwodifferent}, asking the question: ``Does a unique global minimum for chordal cost with any noise magnitude exist for the 3-pose case too?'' By adding orientation information to the landmark, that problem becomes the 3-pose problem considered in this paper. We prove that for the noise-free case, a unique global minimum exists, but provide a counterexample to show that uniqueness of the global minimum does not hold for arbitrary noise in the 3-pose case.
		\item Finally, we search for points that failed to converge to the global minimum in the case of chordal cost. Across all three example problems, only four singleton points are found, each failing to converge due to numerical issues. This is a significant reduction in area compared to the regions of attraction to local minima in the geodesic cost case.
	\end{itemize}

	
	
	The paper is structured as follows: we first define the minimal SLAM problem using geodesic and chordal cost in Section \ref{sec:problemFormulation}, and rewrite them in more convenient formulations. Then, we analyze the number and nature of minima for geodesic and chordal cost in Sections \ref{sec:angularCost} and \ref{sec:chordalCost}, respectively. In Section \ref{sec:exampleAndDiscussion}, we analyze a few examples and compute their regions of attraction to local minima when using geodesic cost. Finally, we conclude with Section \ref{sec:conclusion}.
	
	\section{Two formulations of a 3-pose planar pose-graph SLAM problem}\label{sec:problemFormulation}
	\subsection{Notation and conventions}
	In this paper, we use the semicolon to mean vertical vector concatenation. For an angle $\phi\in\R$, let $R(\phi)\in SO(2)$ be its corresponding rotation matrix.
	\subsection{The 3-pose problem}
	In this paper we consider a 2D pose-graph problem with three poses and three measurements. 
	Let each of the \textit{poses} have a position $p_i\in\R^2$, and an orientation $\phi_i\in\R$ for $i\in\{0,1,2\}$; for short we write $R_i = R(\phi_i)$. Let the vector $P = [p_1;p_2]\in\R^4$ and the the vector $\Phi = [\phi_1;\phi_2]\in\R^2$. In this paper we will treat $p_0 = 0\in\R^2$ and $\phi_0 = 0$ as fixed, so we exclude them from $P$ and $\Phi$. 
	
	Suppose at each pose the robot has taken some measurements from pose $i$ to pose $j$: a relative position $p_{ij}\in\R^2$, and a relative rotation $\phi_{ij}\in\R$; as before, we use the shorthand $R_{ij} = R(\phi_{ij})$. We assume the most ideal case, that each measurement in $p_{ij}$ and $\phi_{ij}$ has variance $\sigma^2$; hence let $\Sigma = \sigma^2 I$, where $I$ is the (square) identity matrix whose size is determined by context. Hence $\Sigma$ is a spherical covariance matrix \cite{Wang2013nonlinearitystructurepoint}. For simplicity, we have assumed that $p_{ij}$ and $\phi_{ij}$ have the same variance. However, the analysis in this paper holds if the position and orientations have different variances.
	
	In this paper's formulation of the 3-pose problem, we assume there are three measurements, resulting in three relative positions $p_{01},p_{12},p_{02}$ and three relative rotations $\phi_{01},\phi_{12},\phi_{02}$. Throughout the paper, we assume that none of the measurements are zero, and that none of the poses are equal to another.
	
	With these definitions, a pose-graph optimization problem can be set up to find robot poses that best satisfy these measurements, according to some cost function. 
	One formulation of the cost function is what we will call the ``geodesic cost'' $F$:
	\begin{align}
	F(P,\Phi) &= \underbrace{\sum_{(i,j)\in\mathcal I}|p_{ij} - R_i\transp(p_j-p_i)|_\Sigma^2}_{F_p(P,\Phi)} \label{eqn:6DimWrapCost}\\
	& + \underbrace{\sum_{(i,j)\in\mathcal I}\left(\frac{\wrap(\phi_{ij}-(\phi_j-\phi_i))}{\sigma}\right)^2}_{F_\phi(\Phi)}\nonumber,
	\end{align} where $|\cdot|_\Sigma$ is the Mahalanobis distance with respect to covariance $\Sigma$, $\mathcal I$ is the set $\{(0,1),(1,2),(0,2)\}$, and $\wrap(\phi)$ returns the angle equivalent to $\phi$ on the interval $[-\pi,\pi)$.
	
	The other cost function we consider in this paper will be called ``chordal cost'' $G$:
	\begin{align}
	G(P,\Phi) = F_p(P,\Phi) + \underbrace{\sum_{(i,j)\in\mathcal I} \frac{1}{2\sigma^2}\|R_iR_{ij}-R_j\|_F^2}_{G_\phi(\Phi)},\label{eqn:chordalCost}
	\end{align} where $\|\cdot\|_F$ is the Frobenius norm of a matrix. The factor of $\frac{1}{2}$ is introduced so $G_\phi(\Phi)$ and $F_\phi(\Phi)$ have the same linearization at the origin, c.f. \cite[Remark 1]{Carlone2015DualitybasedVerification}. The two cost functions $F$ and $G$ are different ways of quantifying the same qualitative idea: they evaluate how well given poses $P,\Phi$ ``match'' the measurements. Then, by minimizing either
	\begin{align}
	&\min_{P,\Phi} F(P,\Phi)\text{, or}\label{opt:orignalWrapCost}\\ 
	&\min_{P,\Phi} G(P,\Phi),\label{opt:originalChordalCost}
	\end{align} $(P,\Phi)$ can be found that explain the measurements well. 
	
	Notice that $F$ and $G$ share the same $F_p(P,\Phi)$, and differ only in $F_\phi(\Phi)$ and $G_\phi(\Phi)$. For our particular problem, $F_p(P,\Phi)$ simplifies to:
	\begin{equation}
	F_p(P,\Phi) = F_p(P,\phi_1),
	\end{equation} since $\phi_0 = 0 \Rightarrow R_0 = I_2$ and $F_p$ does not depend on $R_2$.

	\subsection{Dimensionality reduction via Schur complement}\label{sec:dimensionality reduction}
	The decision space for the optimization problems minimizing $F$ and $G$ is $\dim(P)\cup \dim(\Phi) = 6$. In this subsection we reduce it to two dimensions by noticing that the problem of minimizing $F_p(P,\phi_1)$ can be solved in closed-form given any $\phi_1$. The following lemma will aid us in this \cite{Wang2013nonlinearitystructurepoint}:
	\begin{lemma}\label{lem:linearLeastSquaresSolution}
		The linear least-squares problem of minimizing $|Ax-b|_C^2$ has a unique solution if $A$ has full column rank, and $C > 0$:
		\begin{align}
		x^\star &= \underset{x}{\text{argmin }} |Ax-b|_C^2 = (A\transp C\inv A)\inv A\transp C\inv b\\
		F^\star &= \underset{x}{\text{min }} |Ax-b|_C^2 = b\transp Q b,
		\end{align} where $Q = C\inv - C\inv A(A\transp C\inv A)\inv A\transp C\inv$. \hfill $\square$
	\end{lemma}

	Now, we rewrite $F_p$ for use with Lemma \ref{lem:linearLeastSquaresSolution}:
	\begin{align}
	F_p(P,\phi_1) &= |p_{01} - (p_1-p_0)|_{\Sigma}^2 + |p_{12} - R_1\transp(p_2-p_1)|_{\Sigma}^2\nonumber\\ 
	&\phantom{aa}+ |p_{02} -(p_2-p_0)|_{\Sigma}^2 \\
	&= |AP - z_0 - \bar R(\phi_1)z_1|_C^2 \label{eqn:FpAsLinearLeastSquares},
	\end{align} where $z_0 = [p_{01};0;p_{02}],z_1 = [0;p_{12};0]$, $\bar R(\phi_1)$ is the $6 \times 6$ matrix that is $\bar R(\phi_1) = \diag(R_1^\intercal,R_1^\intercal,R_1^\intercal)$, and 
	\begin{equation}
	A = \begin{bmatrix}
	I_2 & 0\\
	-I_2 & I_2\\
	0 & I_2
	\end{bmatrix}.
	\end{equation} The $6\times6$ matrix $C$ is $\text{diag}(\Sigma,\Sigma,\Sigma)$.
	Hence by Lemma \ref{lem:linearLeastSquaresSolution},
	\begin{align}
	P^\star(\phi_1) &:= \underset{P}{\text{argmin }} F_p(P,\phi_1)\\
	&= (A\transp C\inv A)\inv A\transp C\inv (z_0-\bar R(\phi_1)z_1)\\
	F_p^\star(\phi_1) &:= \underset{P}{\text{min }} F_p(P,\phi_1) = |z_0-\bar R(\phi_1)z_1|^2_Q.
	\end{align} $F_p^\star(\phi_1)$ evaluates to \cite[Theorem 1]{Wang2013nonlinearitystructurepoint}:
	\begin{equation}
	F_p^\star(\phi_1) = c_0 - 2a_0\cos(\phi_1 - \theta_0),
	\end{equation} where the constants $c_0 = z_0\transp Q z_0 + z_1\transp Q z_1$, $a_0 = \sqrt{(z_0\transp Q z_1)^2 + (z_0\transp Q \bar R(\frac{\pi}{2})z_1)^2}$, and $\theta_0 = \text{atan2}(-z_0\transp Q \bar R(\frac{\pi}{2})z_1, -z_0\transp Q z_1)$ are determined by the measurement and covariance data only. Notice that $c_0,a_0 > 0$, since the measurements are assumed to be nonzero.

	Then, if we let 
	\begin{align}
	f(\Phi) &= c_0 - 2a_0\cos(\phi_1 - \theta_0) + F_\phi(\Phi)\\
	g(\Phi) &= c_0 - 2a_0\cos(\phi_1 - \theta_0) + G_\phi(\Phi),
	\end{align}
	then instead of solving \eqref{opt:orignalWrapCost} and \eqref{opt:originalChordalCost}, we can instead solve the \textit{two-dimensional} problems:
	\begin{align}
	\min_{\Phi} f(\Phi) &:= \min_{\Phi} F(P^\star(\phi_1),\Phi)\label{opt:2Dwrap}\\
	\min_{\Phi} g(\Phi) &:= \min_{\Phi} G(P^\star(\phi_1),\Phi).\label{opt:2Dchordal}
	\end{align}
	We will use $f(\Phi)$ and $f(\phi_1,\phi_2)$ interchangeably, and similarly with $g(\Phi)$ and $g(\phi_1,\phi_2)$. 
	The following lemma tells us what minima in $f$ and $g$ imply about minima in $F$ and $G$, which will be used in the proofs of Theorem \ref{thm:localMinimaOnRk} and \ref{thm:minIsUniqueForChordal}.
	
	\begin{lemma}\label{lem:minOfLowDimProblemIsMinOfHighDim}
		Consider the problem of minimizing a function of two variables $f(x_1,x_2)$ with $x_1\in\R^n,x_2\in\R^m$. Suppose also there exists a function $x_2^\star(x_1)$ that for fixed $x_1$,
		\begin{equation}\label{eqn:x2star}
		x_2^\star(x_1) = \underset{x_2}{\text{argmin}}f(x_1,x_2)
		\end{equation} Then, if 
		\begin{equation}
		\label{eqn:x1star}
		x_1^\star = \underset{x_1}{\text{argmin}}f(x_1,x_2^\star(x_1)),
		\end{equation} then 
		\begin{equation}\label{eqn:x1starx2star}
		(x_1^\star,x_2^\star(x_1^\star)) = \underset{x_1,x_2}{\text{argmin}}f(x_1,x_2).
		\end{equation} If additionally $f$ is known to have a unique minimum, then $(x_1^\star,x_2^\star(x_1^\star))$ is its unique minimum.
	\end{lemma}
	
	\begin{proof}
		For any $x_1\in\R^n,x_2\in\R^m$, by \eqref{eqn:x1star} and \eqref{eqn:x2star},
		\begin{equation}
		f(x_1^\star,x_2^\star(x_1^\star)) \leq f(x_1,x_2^\star(x_1)) \leq f(x_1,x_2),
		\end{equation} and hence \eqref{eqn:x1starx2star}. Uniqueness of the minimum on $f$ yields uniqueness of $(x_1^\star,x_2^\star(x_1^\star))$.
	\end{proof}

	
	\section{Analyzing local minima of ``geodesic cost''}\label{sec:angularCost}
	In this section, we consider the pose-graph optimization problem \eqref{opt:2Dwrap}. We further reduce it to a set of one-dimensional optimization problems, and use these 1D optimization problems to analyze the local and global minima of \eqref{opt:2Dwrap} and the original problem  \eqref{opt:orignalWrapCost}.

	\subsection{Representing \eqref{opt:2Dwrap} as three 1D optimization problems}
	Let $\mathcal S$ be the $2\pi \times 2\pi$ square with $\phi_1\in[\phi_{01}-\pi,\phi_{01}+\pi]$ and $\phi_2\in[\phi_{02}-\pi,\phi_{02}+\pi]$. Because wrap() is $2\pi$-periodic in $\phi_1$ and $\phi_2$, it suffices to consider only $\mathcal S$ when analyzing $F_\phi$. Figure \ref{fig:wrapsurface} shows a surface plot of $F_\phi(\Phi)$.
	
	\begin{figure*} 
		\centering
		\subfloat[Surface plot of $F_\phi(\Phi)$ on $\mathcal S$ for $\Sigma = I$.\label{fig:wrapsurface}]{%
			\includegraphics[height=0.22\textheight]{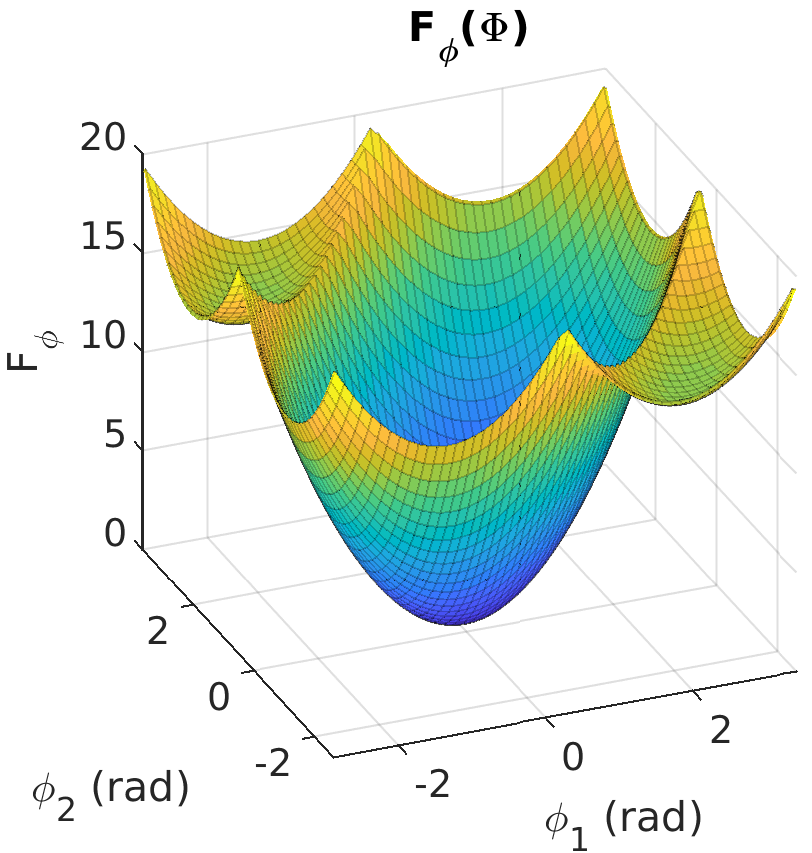}}
		\hfill
		\subfloat[Surface plot of $f(\Phi)$ on $\mathcal S$. \label{fig:surfplotoptimal2d}]{%
			\includegraphics[height=0.22\textheight]{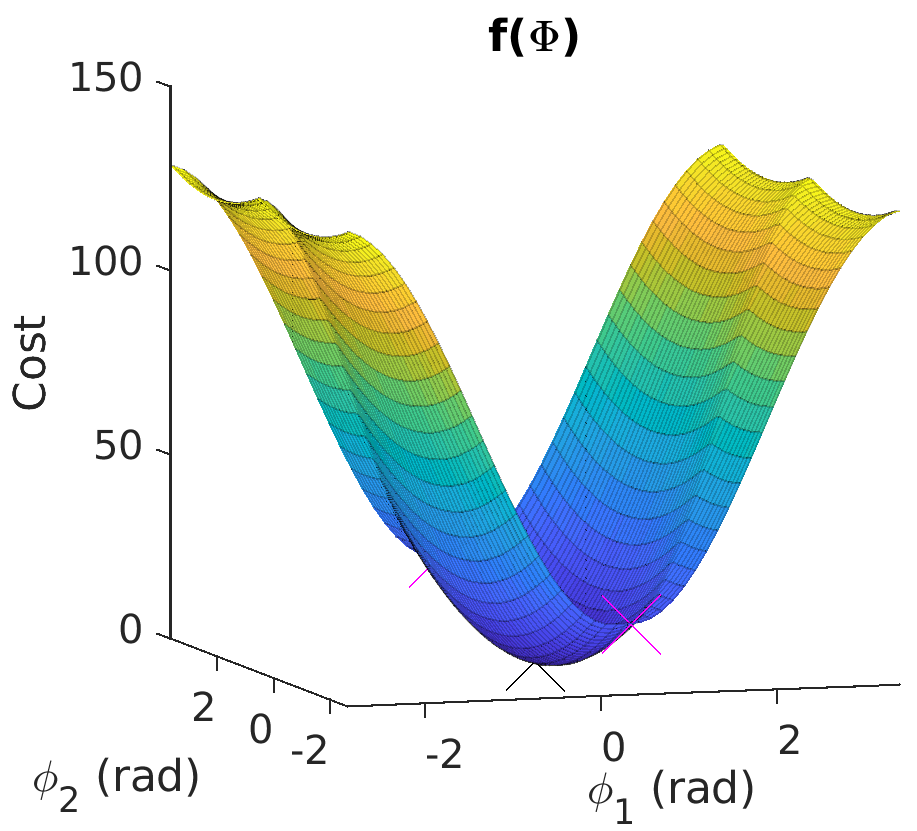}}
		\hfill
		\subfloat[1D problem costs $f_{1,k}(\phi_1)$. \label{fig:opt1dcurves}]{%
			\includegraphics[width=0.23\textheight]{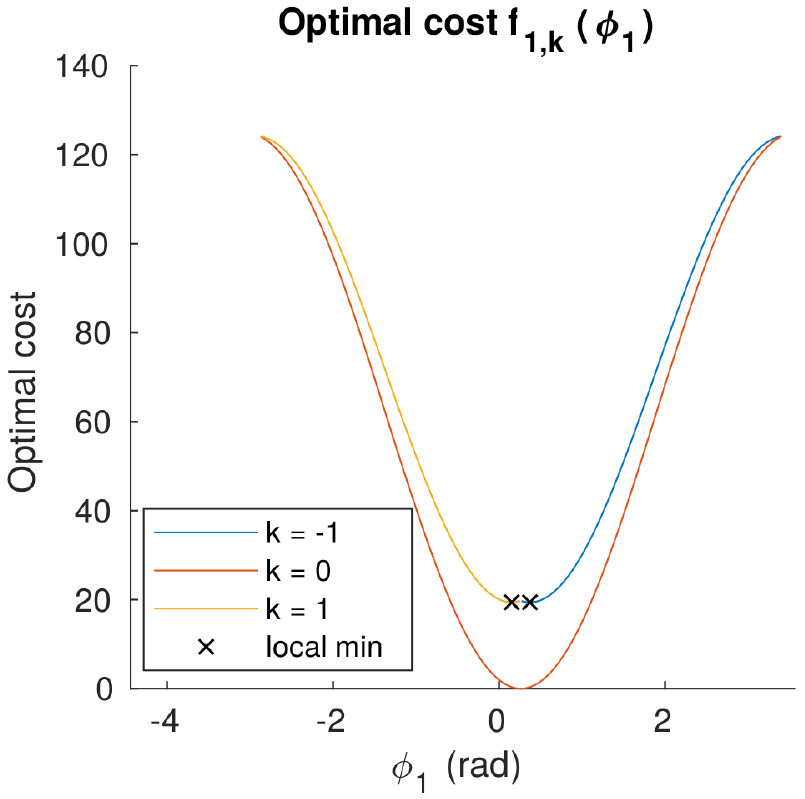}}
		
		\caption{\textbf{Plots from the noise-free 3-pose example problem in Section \ref{sec:exampleAndDiscussion} using geodesic cost.} (a) Notice that on each square $\mathcal S$, there are three minima of $F_\phi(\Phi)$, one on each region $\mathcal R_k$ (c.f. Figure \ref{fig:regions}). (b) Now for $f(\Phi)$: suboptimal local minima are marked by pink x's; we show their existence in Theorem \ref{thm:localMinimaOnRk}. The global minimum in $\mathcal S$ is marked by a black x. (c) 1D optimal costs $f_{1,k}(\phi_1)$ plotted on their domain of definition in $\phi_1$. Notice minima exist for $k = \pm 1$ with cost approximately equal to 20, which are marked with black `x's. }
		\label{fig:group2}
	\end{figure*}

	On the region $\mathcal S$, the function $F_\phi(\Phi)$ is:
	\begin{align}
	F_\phi(\Phi) &= \frac{1}{\sigma^2}\left((\phi_{01}-\phi_1)^2+\wrap(\phi_{12}-(\phi_2-\phi_1))^2\right. \nonumber\\ & \quad + (\phi_{02}-\phi_2)^2\left.\right)
	\end{align}
	
	To keep notation compact, we use the shorthand
	\begin{equation}
	\xi(\Phi) = \phi_{12}-(\phi_2-\phi_1).
	\end{equation} Then, on $\mathcal S$, $\wrap(\xi(\Phi))$ can be replaced by:
	\begin{equation}
	\wrap(\xi(\Phi)) = \begin{cases}
	\xi(\Phi) & \text{if } \xi(\Phi) \in (-\pi,\pi)\\
	\xi(\Phi)-2\pi & \text{if } \xi(\Phi)> \pi\\
	\xi(\Phi)+2\pi & \text{if } \xi(\Phi) < -\pi\\
	\end{cases}
	\end{equation}

	Hence we can rewrite this as $\wrap(\xi(\Phi)) = \xi(\Phi) \pm 2k\pi$ for $k\in\{-1,0,1\}$ on some appropriate regions $\mathcal R_k$ of $\mathcal S$. This results in a natural subdivision of $\mathcal S$ into three regions $\mathcal R_k$ (see Figure \ref{fig:regions}): for $\xi(\Phi) > \pi$, $k = -1$, which corresponds to the (open) lower right triangle, for $\xi(\Phi) < -\pi$, $k = 1$, which corresponds to the (open) upper left triangle, and $\xi(\Phi)\in(-\pi,\pi)$ for $k = 0$, the middle region (also open). Notice we have included in $\mathcal R_k$ any points on the non-differentiable boundary where $\xi(\Phi) = \pm\pi$.
	
	Then, for each $k\in\{-1,0,1\}$, for $\Phi\in\mathcal R_k$, $F_\phi(\Phi)$ can be rewritten:
	\begin{align}
	F_\phi(\Phi) = F_{\phi,k}(\Phi) := & \frac{1}{\sigma^2}\left(\right.(\phi_{01}-\phi_1)^2 + (\phi_{02}-\phi_2)^2 \nonumber\\&+ (\xi(\Phi) + 2k\pi)^2\left.\right)\label{eqn:wrapSurrogateCost}.
	\end{align} 
	
	For each $k$, it can be seen that \eqref{eqn:wrapSurrogateCost} is a least-squares cost function in $\phi_2$. That is, for any given $\phi_1$ and $k$, we can find the optimal $\phi_2$ that minimizes $F_{\phi,k}$ by again using Lemma \ref{lem:linearLeastSquaresSolution}. This reduces the 2D optimization problem of minimizing $f_k(\Phi)$ to minimizing a 1D problem.  We rewrite $F_{\phi,k}(\Phi)$ as:
	\begin{equation}
	F_{\phi,k}(\phi_1,\phi_2) = |A_2 \phi_2 + Z_k + A_1\phi_1|_{C_\phi},
	\end{equation} where the $3\times1$ column vectors $A_2 = [0; -1; -1], Z_k = [\phi_{01};\phi_{12} + 2k\pi; \phi_{02}]$, and $A_1 = [-1; 1;0]$. The $3\times 3$ matrix $C_\phi$ is equal to $\text{diag}(\sigma^2,\sigma^2,\sigma^2)$. Hence the angular component of the cost function can be written:
	\begin{align}
	F_{\phi,k,1}(\phi_1) &:= \min_{\phi_1,k} F_{\phi,k}(\phi_1,\phi_{2,k}^\star(\phi_1))\nonumber\\
	&= (Z_k+A_1\phi_1)\transp Q_2(Z_k+A_1\phi_1), 
	\end{align} where $Q_2 = C_\phi\inv - C_\phi\inv A_2(A_2\transp C_\phi\inv A_2)\inv A_2\transp C_\phi\inv$, and
	\begin{equation}
	\phi_{2,k}^\star(\phi_1) = \frac{1}{2}(\phi_1 + 2k\pi + \phi_{12} + \phi_{02}).
	\end{equation}
	Hence $f(\Phi)$ can also be re-written as a set of one-dimensional cost functions:
	\begin{align}
	f_{1,k}(\phi_1) &:= \min_{\phi_1}f(\phi_1,\phi_{2,k}^\star(\phi_1)) =  c_0 - 2a_0\cos(\phi_1-\theta_0) +\nonumber\\ &(Z_k+A_1\phi_1)\transp Q_2(Z_k+A_1\phi_1). \label{eqn:1dSurrogateCost}
	\end{align} For each $k$, this is obviously smooth, and has derivatives:
	\begin{align}
	f'_{1,k}(\phi_1) &= 2a_0\sin(\phi_1-\theta_0) + 2(Z_k+\phi_1A_1)\transp Q_2 A_1,\\
	f''_{1,k}(\phi_1) &= 2a_0\cos(\phi_1-\theta_0) + 2A_1\transp Q_2 A_1.
	\end{align}

	Hence we have reduced the dimension of the optimization problem from 2D in $f(\Phi)$ to a set of three 1D optimization problems in $f_{1,k}(\phi_1)$. Figure \ref{fig:opt1dcurves} shows the one-dimensional $f_{1,k}$ for each region for the example problem in Section \ref{sec:exampleAndDiscussion}.
	
	We also define the 2- and 6-dimensional cost functions for each $k$.	In place of $f(\Phi)$, we consider three corresponding 2D problems:
	\begin{equation}
	f_k(\Phi) = c_0 - 2a_0\cos(\phi_1-\phi_{01}) + F_{\phi,k}(\Phi),
	\end{equation} and in place of $F(P,\Phi)$,
	\begin{equation}
	F_k(P,\Phi) = F_P(P,\phi_1) + F_{\phi,k}(\Phi).
	\end{equation} 
	\begin{remark}\label{rem:minfkDoesntMeanMinf}
		Note however that even if some $\Phi^\star$ is a global minimum of $f_k(\Phi)$, this does not necessarily imply it is a global minimum of $f(\Phi)$. This is because $f_k(\Phi)=f(\Phi)$ only on $\mathcal R_k$; $f(\Phi)$ may  be less than $f_k(\Phi^\star)$ outside $\mathcal R_k$.
	\end{remark}

	\subsection{Main result: Existence of multiple local minima of $F$}
	Now that we have represented $f(\Phi)$ as a triplet of 1D problems $f_{1,k}(\phi_1)$, we use them to analyze $f(\Phi)$ and $F(P,\Phi)$. 
	
	In this section, we assume that the measurements are ``perfect''; that is, 
	\begin{equation}\label{eqn:measurementsArePerfect}
	\phi_{01} + \phi_{12} = \phi_{02}.
	\end{equation} When the measurements are perfect, $\Phi^\star$, the global minimum on $\mathcal S$, should match measurements exactly: 
	$\Phi^\star = [\phi_{01};\phi_{02}]$. This can be seen by checking that $f(\Phi^\star) = 0$. We are more interested in proving the existence of suboptimal local minima. 
	
	We claim that even in the ideal case of spherical covariance and perfect measurements, there are multiple suboptimal local minima of $f$ and $F$. The proofs contain only elementary linear algebra and vector calculus, and have been relegated to the appendix.
	\begin{lemma}\label{lem:noGlobalMinimaInR_pm1}
		Assume \eqref{eqn:measurementsArePerfect} holds. Then, there are no global minima of $f(\Phi)$ in $\mathcal R_{\pm 1} = \mathcal R_1 \cup \mathcal{R}_{-1}$. 
	\end{lemma}

\begin{theorem}\label{thm:localMinimaOnRk}
	Assume that the measurements are perfect, i.e. \eqref{eqn:measurementsArePerfect}. Then, $f(\Phi)$ has at least two suboptimal local minima on $\mathcal S$, one in $\mathcal R_1$, and the other in $\mathcal R_{-1}$. Each of these correspond to (suboptimal) local minima of $F(P,\Phi)$.
\end{theorem}

However, in practice, \eqref{eqn:measurementsArePerfect} does not hold, and there is usually some inconsistency in the measurements:
\begin{equation}\label{eqn:measurementsAreImperfect}
\phi_{01} + \phi_{12} = \phi_{02} + \varepsilon.
\end{equation}
It is easy to show that the boundaries of the regions $\mathcal R_k$ vary with $\varepsilon$; see Figure \ref{fig:roaGroup} for some examples. In the event that $\varepsilon$ is ``large'', the number of minima on $f(\Phi)$ may change.
\begin{remark}\label{rem:localMinMayDisappear}
	For ``large enough'' measurement mismatch $\varepsilon$, there may not exist minima on the open set $\mathcal R_{1}$.
	In the proof of Theorem \ref{thm:localMinimaOnRk}, suppose $a_0$ and $\sigma$ fixed and $\frac{3}{2a_0\sigma^2}\in(0,1)$, i.e. we are in case ``b'' in the proof. 
	Then, Theorem \ref{thm:localMinimaOnRk} relies on finding a minimum in the interval $\phi_1\in[\phi_{01}-\alpha,\phi_{01}]$, where $\alpha \in (0,\pi)$. However, with $\varepsilon \neq 0$, the interval of $\phi_1$ on which $\mathcal R_{\pm1}$ is defined shrinks. If it shrinks enough so that the minimum of $f_{1,k}$ found through Theorem \ref{thm:localMinimaOnRk} is not actually in $\mathcal R_{1}$, there will not be a minimum in $\mathcal R_1$. The same logic applies to $\mathcal R_{-1}$.
\end{remark}

In conclusion, the use of geodesic distance in the cost function $f(\Phi)$ results in a nonsmooth cost function that has multiple suboptimal local minima, even in the case of perfect measurements. 

\section{Analyzing minima of chordal cost}\label{sec:chordalCost}
	In this section we investigate the minima of optimization problem \eqref{opt:2Dchordal} and \eqref{opt:originalChordalCost}. In contrast to the previous section, we show that if measurements are perfect, the use of chordal distance yields a unique global minimum, and no suboptimal global minima. 
	
	Expanding and simplifying $G_\phi(\Phi)$ from \eqref{eqn:chordalCost}:
	\begin{equation}
	G_\phi(\Phi) = 2(3 - \cos(\phi_{01}-\phi_1) - \cos(\phi_{02}-\phi_2) - \cos(\xi(\Phi))).
	\end{equation}
	Figure \ref{fig:chordalCostAngleOnly} shows $G_\phi(\Phi)$ for the example problem considered in Figure \ref{fig:localminimumslamsolution}.
	Hence \eqref{opt:2Dchordal} can be rewritten:
	\begin{align}
	g(\Phi) = & (c_0 - 6) -2\Big((a_0+1)\cos(\phi_{01}-\phi_1) \nonumber\\ &+ \cos(\phi_{02}-\phi_2) + \cos(\xi(\Phi)) \Big).
	\end{align} Figure \ref{fig:chordalCost} shows $g(\Phi)$ for our example.
	We will also make use of the $J$, the Jacobian of $g(\Phi)$:
	\begin{equation}\label{eqn:jacobian}
	J(\Phi)\transp = 2\begin{bmatrix}
	-(a_0+1)\sin(\phi_{01}-\phi_1)+\sin(\xi(\Phi))\\
	-\sin(\phi_{02}-\phi_2) - \sin(\xi(\Phi))
	\end{bmatrix}
	\end{equation} and the Hessian
	\begin{align}
	&H(\Phi) = \nonumber\\ &2\left[\begin{smallmatrix}
	(a_0+1)\cos(\phi_{01}-\phi_1)+\cos(\xi(\Phi)) & -\cos(\xi(\Phi))\\
	-\cos(\xi(\Phi))& \cos(\phi_{02}-\phi_2) + \cos(\xi(\Phi))
	\end{smallmatrix}\right].
	\end{align}
	\subsection{Main result: Unique existence of global minimum of $G$} 
	The main claim of this section is the following theorems. Again, the proofs are elementary and have been relegated to the appendix.
	\begin{theorem}\label{thm:minIsUniqueForChordal}
		Assume that the measurements are perfect, i.e. \eqref{eqn:measurementsArePerfect} holds. Then, $\Phi = [\phi_{01};\phi_{02}]$ is the unique minimum on $\mathcal S$.
	\end{theorem}

	However, for the case of imperfect measurements, this is no longer true. In the worst case, for $\varepsilon = \pi$, we have two distinct global minima on $\mathcal S$:
	\begin{theorem}\label{thm:noisyCaseChordal}
		If $\varepsilon = (2n+1)\pi$ with $n\in\mathbb{Z}$, multiple distinct global minima of $g(\Phi)$ exist on $\mathcal S$. 
	\end{theorem}
	
	Hence it cannot be true that a unique global minimum exists for arbitrary noise. This is a significant difference of the 3-pose problem compared to the ``one-step'' problem in \cite{Wang2018Comparisontwodifferent}, which had a unique minimum for any noise magnitude.

	\begin{figure} 
\centering
	\subfloat[Cost for chordal distance $G_\phi(\Phi)$. \label{fig:chordalCostAngleOnly}]{%
		\includegraphics[height=0.2\textheight]{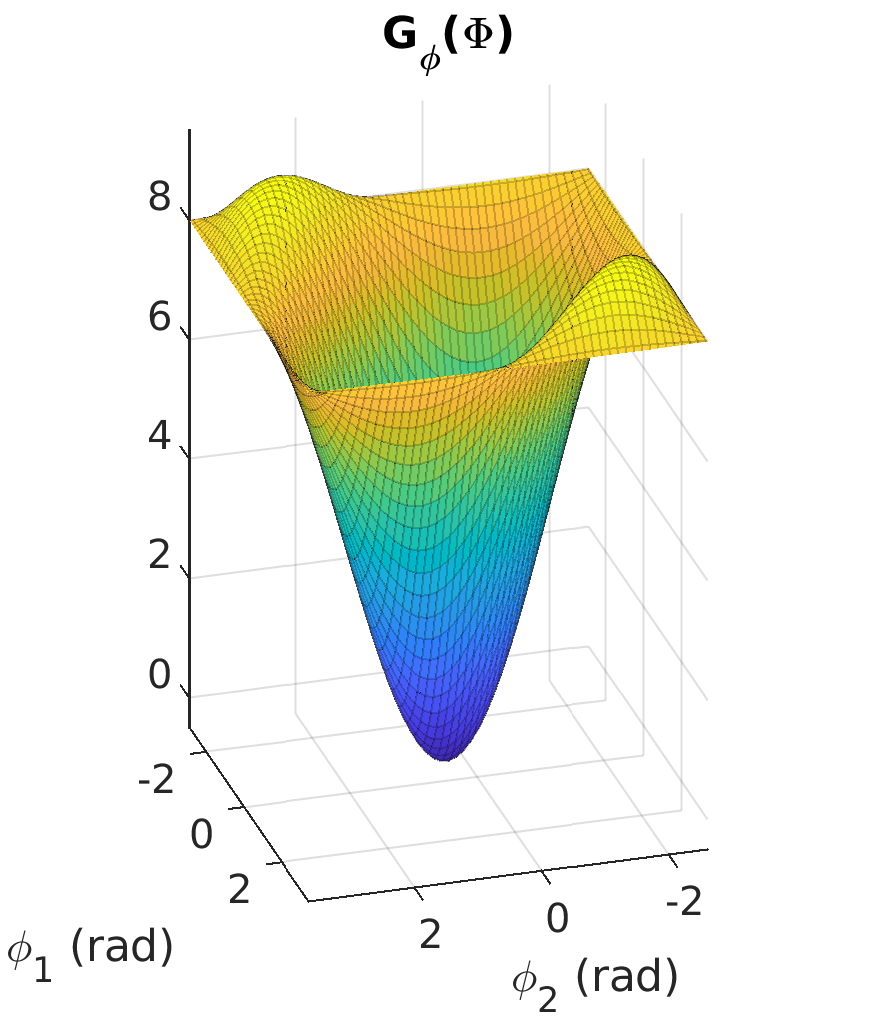}}
	\subfloat[Chordal cost $g(\Phi)$ on $\mathcal S$. \label{fig:chordalCost}]{%
		\includegraphics[height=0.2\textheight]{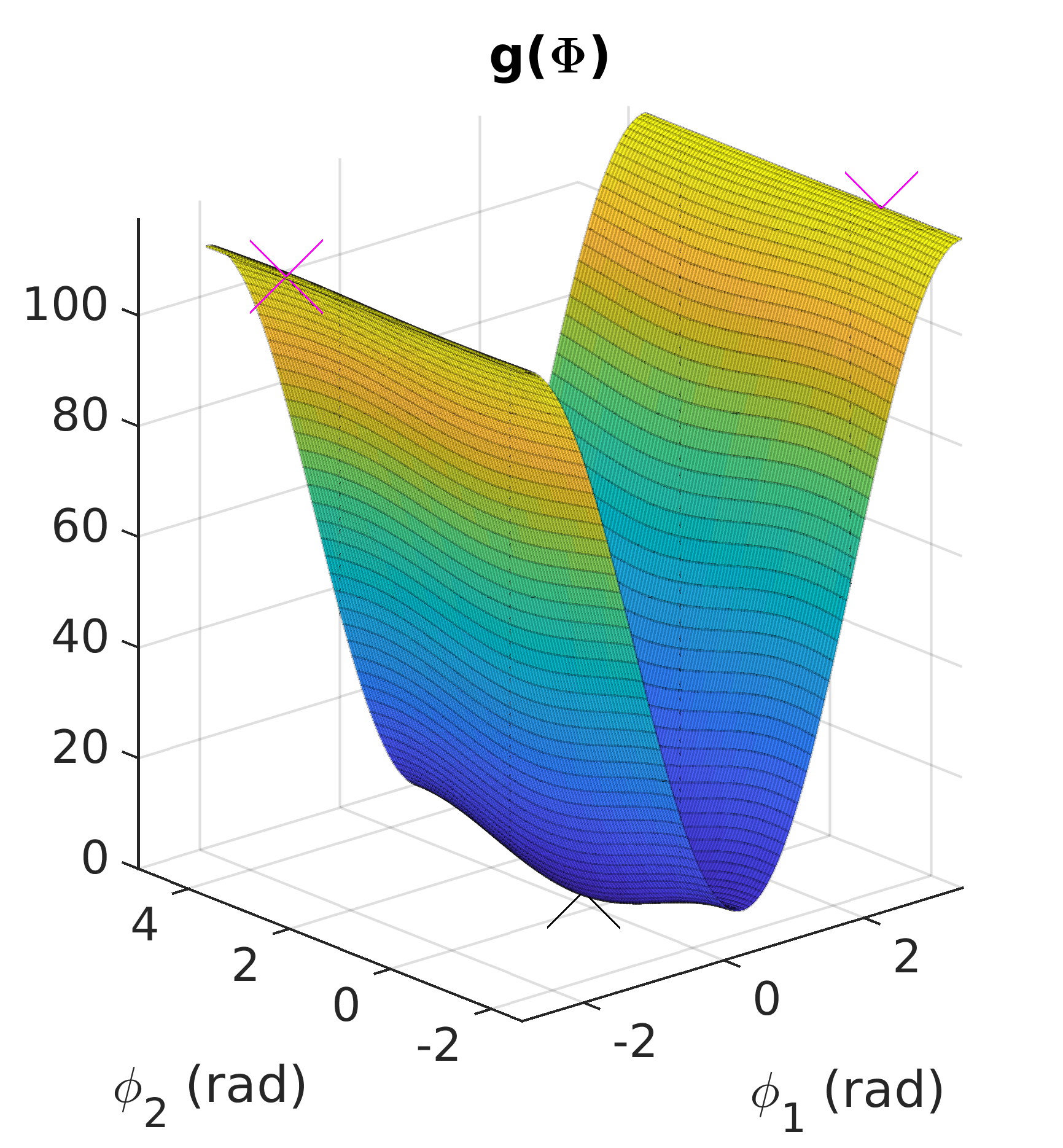}}
	
	\caption{ \textbf{Plots from the noise-free 3-pose example problem in Section \ref{sec:chordalCost} using chordal cost.} (a) 
		Compared to $F_\phi(\Phi)$ in Figure \ref{fig:wrapsurface}, $G_\phi(\Phi)$ is much more well-behaved. (b) Compared to Figure \ref{fig:surfplotoptimal2d}, $G_\phi(\Phi)$ is smooth and has a unique minimum exists on $\mathcal S$. Maxima are marked with pink `x's, the unique global minimum is marked with a black `x'. }
	\label{fig:group3}
\end{figure}	

\begin{figure}
	\hfill\subfloat{\includegraphics[width=0.5\columnwidth]{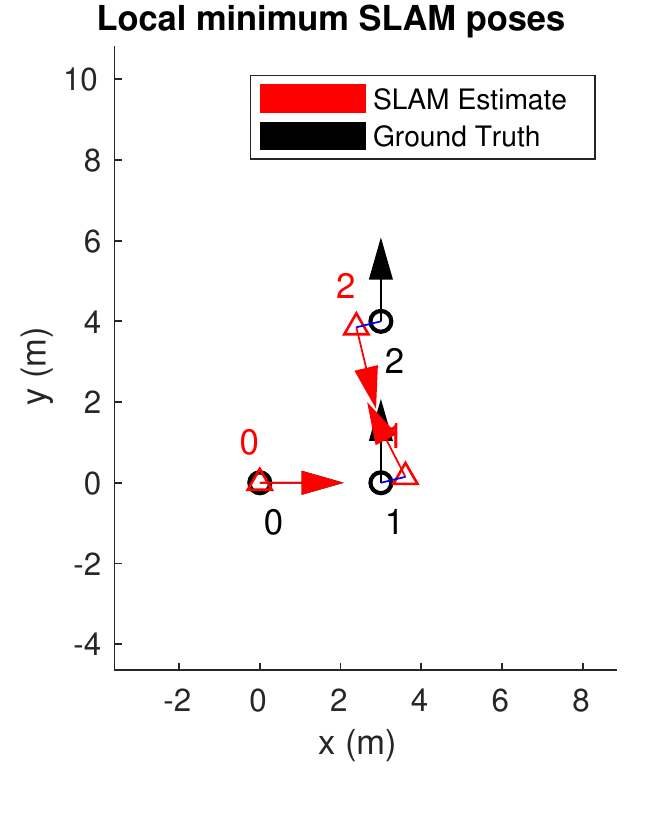}}
	\subfloat{%
		\includegraphics[width=0.5\columnwidth]{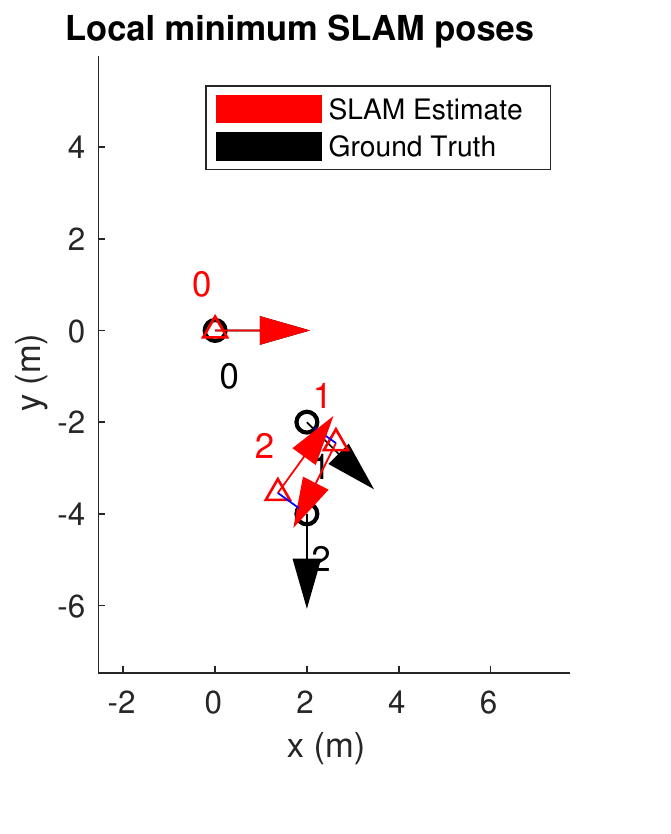}}
	\hfill
	\caption{Poses at local minima of $f(\Phi)$ for example problems 2 (left) and 3 (right), which incorporate noisy, imperfect measurements (c.f. Table \ref{tab:exampleProblems}).}
	\label{fig:groupNoisySlamSolutions}
\end{figure}

\begin{figure*}
	\centering
\includegraphics[height=0.4\textheight]{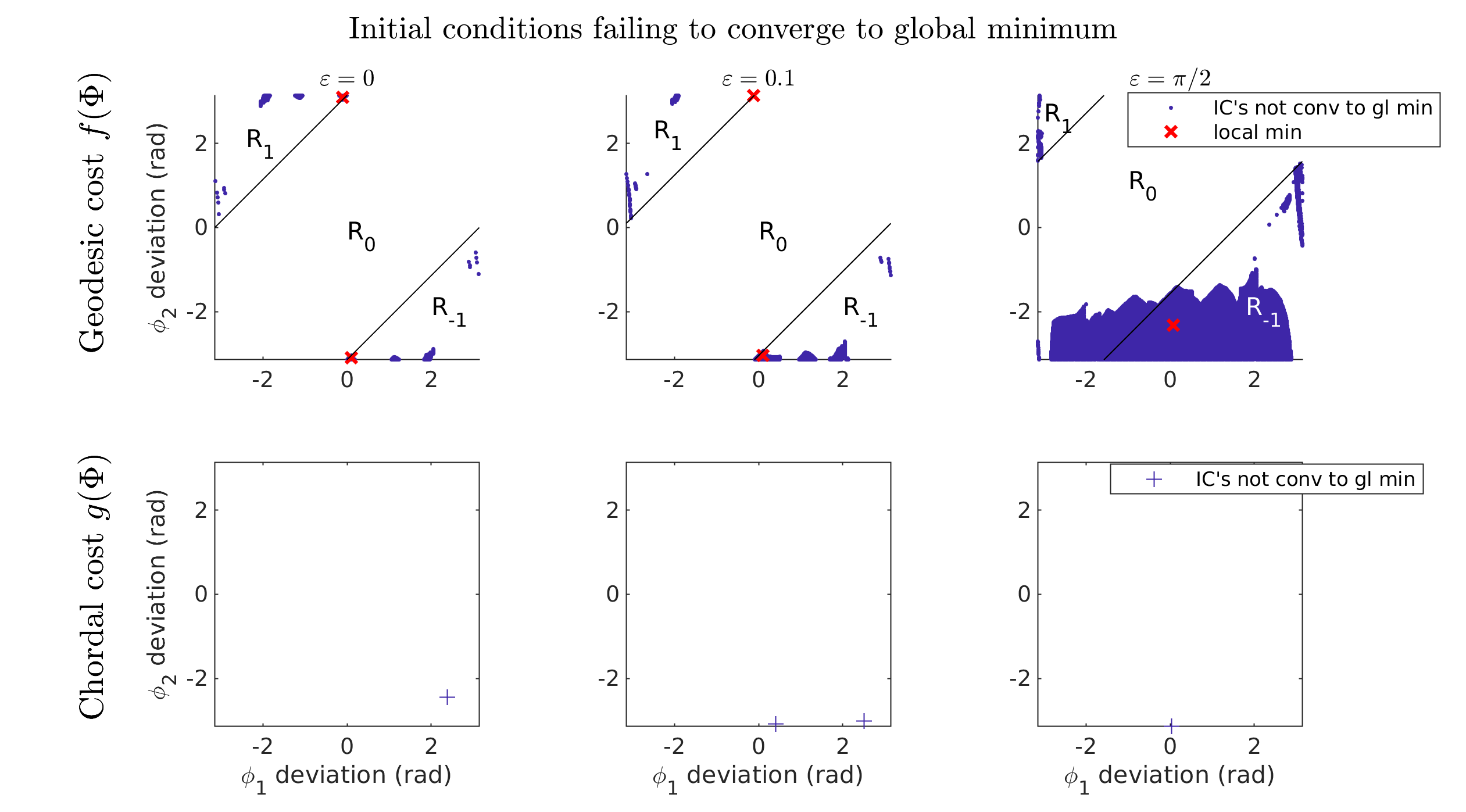}
	\caption{
	Sampled initial conditions (IC's) that did not converge to the global minimum for each example problem in the first three columns of Table \ref{tab:exampleProblems}. (Top row) All IC's that did not converge to the global minimum converged instead to a local minimum. The diagonal lines indicate the boundaries of $\mathcal R_k$. (Bottom row) IC's not converging to the global minimum are marked as $+$ for better visibility; these IC's all failed to converge due to numerical issues.
}
	\label{fig:roaGroup}
\end{figure*}
	
	
	\section{Examples and Discussion}\label{sec:exampleAndDiscussion}
	
In this section we consider several numerical examples, firstly to illustrate the results of Theorem \ref{thm:localMinimaOnRk} in the case of perfect measurements, and secondly to give more intuition about the noisy case, which is far more common in practice. 

Theorem \ref{thm:localMinimaOnRk} applies to any planar 3-pose problem with perfect measurements, and not just a single, contrived example. To emphasize this, we consider three problems with three different ground truths. To investigate the effect of noise, we have applied three levels of noise, one to each example problem: $\varepsilon \in \{0, 0.1, \pi/2\}$. Noise of $\varepsilon/3$ was added to each orientation measurement; no noise was added to position measurements. Figure \ref{fig:groupNoisySlamSolutions} shows poses corresponding to the local minima of the two noisy 3-pose problems, and Table \ref{tab:exampleProblems} summarizes these problems. We have also verified the existence of these local minima separately using MATLAB's Navigation Toolbox.

\begin{table}[]
	\centering
	\caption{Percentage of sampled points converging to a local minimum}
	\begin{tabular}{|c|l|l|l|}
		\hline
		\thead{Example\\Problem} & Ground truth $\Phi$ (rad)& $\varepsilon$  (rad) & \thead{\% IC's on $\mathcal S$ converging\\ to local min for $f(\Phi)$}\\
		\hline
		1       & $[\pi/12; \pi/6]$            & 0           & 0.2\%   \\
		2       & $[\pi/2; \pi/2]$             & 0.1         & 0.4\%   \\
		3       & $[-\pi/4; -\pi/2]$           & $\pi/2$     & 19.8\%   \\
\hline          
	\end{tabular}
\label{tab:exampleProblems}
\end{table}

 
 Figure \ref{fig:roaGroup} shows plots of initial conditions (IC's) on $\mathcal S$ that failed to converge to the global minimum; both chordal and geodesic cost were considered for each example problem in Table \ref{tab:exampleProblems}. Even though the problems have different ground truth poses, comparison still makes sense, since the positions have been ``optimized out'' (c.f. Section \ref{sec:dimensionality reduction}), and the $x$- and $y$- axes in these plots are \textit{deviations} of $\phi_1$ and $\phi_2$. For each problem, a uniform $500\times500$ grid of $\mathcal S$ was constructed. Each grid point was used as the initial condition for MATLAB's \texttt{fminunc} solver. If a global minimum was reached, the grid point was omitted from the plot; otherwise, it was plotted. Hence the top row of Figure \ref{fig:roaGroup} shows an approximation of the region of attraction (ROA) to suboptimal local minima for each problem.  While we have only shown that a finite number of sampled IC's converged to a local minimum, it seems reasonable due to smoothness that all points on some continuous area ``in between'' these sampled IC's will also converge to the local minimum (obviously, this depends on the choice of solver).


The top row of Figure \ref{fig:roaGroup} yields several interesting conclusions. For $\varepsilon = 0$ and $0.1$, the sampled IC's in $\mathcal R_0$ always converged to the global minimum. This is consistent with the experience of many users that although wrap() is used in the cost function, good results are obtained. A common method for initializing $\Phi$ for consecutive poses is to use odometry; if odometry measurements are reasonably accurate, the initial conditions are likely to be in $\mathcal R_0$, i.e. the linear region of wrap(). This is also consistent with the conclusions in \cite{Wang2013StructureNonlinearitiesPose}, namely that in the noise-free case, a unique minimum that is globally optimal exists if wrap() is assumed to be the identity.

However, for each problem, in the case of geodesic cost, there were some IC's that failed to converge to the global minimum (shown in blue). All of these IC's converged successfully instead to a local minimum. In the case of $\varepsilon = \pi/2$, the local minimum in the top left region $\mathcal R_1$ disappeared (c.f. Remark \ref{rem:localMinMayDisappear}). However, while the total number of suboptimal local minima decreased, the region of attraction to the other local minimum is enormous; almost 20\% of the initial conditions on $\mathcal S$ converged to it.


The bottom row of Figure \ref{fig:roaGroup} shows the same investigation applied to chordal cost $g(\Phi)$. Clearly, many fewer IC's fail to converge to the global minimum. Even though Theorem \ref{thm:minIsUniqueForChordal} guarantees a unique global minimum on $\mathcal S$, there were several singleton initial conditions across the $\varepsilon$ that failed to converge to the global minimum. The result of \texttt{fminunc} for each of these initial conditions had a large gradient (around 20), and had Hessians with condition number on the order of $10^3$, suggesting numerical issues pertaining to the choice of solver and tolerances. 

We emphasize that the \textit{nature} of IC's failing to converge to the global minimum for $g(\Phi)$ is different from those in $f(\Phi)$, which converged successfully, albeit to suboptimal local minima. While some issues pertaining to numerical solvers persist, it is clear from Figure \ref{fig:roaGroup} that $g(\Phi)$ has considerable advantages over $f(\Phi)$ when it comes to convergence.

\section{Conclusion}\label{sec:conclusion}
In this paper, we have shown that for a minimal pose-graph problem, even in the case of ideal measurements, the use of geodesic distance in the cost function results in multiple suboptimal local minima. For several numerical examples, we give evidence that the regions of attraction to these local minima are of nonzero measure, and show that some of these regions of attraction increase in size as noise is added. 

In contrast, under the same idealized conditions, the use of the chordal cost instead of geodesic cost yields a unique global minimum, up to periodicity. In our examples, which have various total noise magnitudes up to $\varepsilon = \pi/2$, the region of attraction of the global minimum is shown to be the whole of $\mathcal S$, except for one or two points due to numerical issues. However, for extremely large noise $(\varepsilon = \pi)$, we show that multiple distinct global minima exist, even for chordal cost. 

While we cannot claim our results apply directly to larger problems, the existence of these regions of attraction due to geodesic cost for this ideal, minimal problem suggests that similar regions may exist for larger problems.
Going forward, a clear future direction would be to extend to the 3D and the $n$-pose 2D case. Also, investigating the connection to Lie-algebraic methods, which share the benefits of using chordal cost, would be a valuable addition to our understanding of the fundamental nature of pose-graph SLAM problems.


\section*{Appendix}
\begin{shownto}{ral}
	In this appendix we give short proofs to save space. The full proofs are available at: http://arxiv.org/abs/1911.05734
\end{shownto}

\begin{proof}[Lemma \ref{lem:noGlobalMinimaInR_pm1}]
	We first consider the case $k = 1$, aiming to show that for every $\phi_1 \in(\phi_{01}-\pi,\phi_{01})$, $f_{1,0}(\phi_1) - f_{1,1}(\phi_1) < 0$. 
	Consider the difference in the 1D optimal costs $f_{1,k}(\phi_1)$ on $\mathcal R_0$ and $\mathcal R_1$:
	\begin{align}
	f_{1,0}(\phi_1)-f_{1,1}(\phi_1) &= -\frac{2\pi}{\sigma^2}(\pi + \phi_1 + \phi_{12}-\phi_{02})\nonumber\\
	&= -\frac{2\pi}{\sigma^2}(\pi + \phi_1 - \phi_{01}),\label{eqn:costR1-costR0}
	\end{align} since measurements are perfect. For $\phi_1 \in [\phi_{01}-\pi,\phi_{01}]$, this is negative; therefore $f_{0,1}(\phi_1) - f_{1,1}(\phi_1) < 0$ and hence no global minima exist in $\mathcal R_1$. The proof for $\mathcal R_{-1}$ is very similar and therefore has been omitted.
\end{proof}
\begin{shownto}{arxiv}
The following propositions will help us in Theorem \ref{thm:localMinimaOnRk}.\end{shownto}
	\begin{proposition}\label{prop:theta0=phi01}
	$\theta_0 = \phi_{01}$. 
\end{proposition}
\begin{proof}
\begin{shownto}{ral}
	This can be shown through somewhat arduous but simple plane geometry, and is omitted to save space.
	\end{shownto}
\begin{shownto}{arxiv}
We know that $\theta_0 = \text{atan2}(-z_0^\intercal Q \bar R(\frac{\pi}{2})z_1, -z_0^\intercal Q z_1)$. We aim to show that $\phi_{01}$ also equals this expression, so that $\phi_{01}=\theta_0$ on $\mathcal S$.

Now, $Q$ can be directly evaluated:
\begin{equation}
Q = \frac{1}{3\sigma^2}\begin{bmatrix}
I_2 & I_2 & -I_2\\
I_2 & I_2 & -I_2\\
-I_2 & -I_2 & I_2
\end{bmatrix},
\end{equation} where $I_2$ is the $2\times 2$ identity matrix. Recall also that $\bar R(\pi/2) = \diag(R(\pi/2)^\intercal,R(\pi/2)^\intercal, R(\pi/2)^\intercal)$ (c.f. \eqref{eqn:FpAsLinearLeastSquares}). Then, let $x_{ij}$ and $y_{ij}$ be position measurements in the $x$- and $y$-directions between poses $i$ and $j$, in the frame of pose $i$.
Finally, letting $\Delta x = x_{02}-x_{01}$ and $\Delta y = y_{02} - y_{01}$, we can directly evaluate $-z_0^\intercal Q\bar R(\frac{\pi}{2})z_1$ and $-z_0^\intercal Q z_1$:
\begin{align}
-z_0^\intercal Q \bar R\left(\frac{\pi}{2}\right)z_1 &= \frac{1}{3\sigma^2}\left(y_{12}\Delta x - x_{12}\Delta y\right),\label{proof:num}\\
-z_0^\intercal Q z_1 &= \frac{1}{3\sigma^2}\left(x_{12}\Delta x + y_{12}\Delta y\right).\label{proof:den}
\end{align}
Now we turn our attention to $\phi_{01}$, and show that it can also be written $\phi_{01} = \text{atan2}(-z_0^\intercal Q \bar R(\frac{\pi}{2})z_1, -z_0^\intercal Q z_1)$. Figure \ref{fig:prooffigure} shows a diagram of the geometry used to write $\phi_{01}$ in this way. From Figure \ref{fig:prooffigure}, it is evident that 
\begin{align*}
\phi_{01} &= \pi - \alpha + \beta,\\
&= \atanTwo(0, -1) - \atanTwo(\Delta y, \Delta x) + \atanTwo(y_{12},x_{12})\\
&= \atanTwo(0,-1) + \atanTwo(x_{12}\Delta y - y_{12}\Delta x, x_{12}\Delta x + y_{12}\Delta y)\\
&= \atanTwo(y_{12}\Delta x - x_{12}\Delta y, x_{12}\Delta x + y_{12}\Delta y)\\
&= \theta_0,
\end{align*} by \eqref{proof:num} and \eqref{proof:den}, and the fact that multiplying both of them by $3\sigma^2$ does not change the value of the two-argument arctangent.
\begin{figure}
	\centering
	\includegraphics[width=0.7\linewidth]{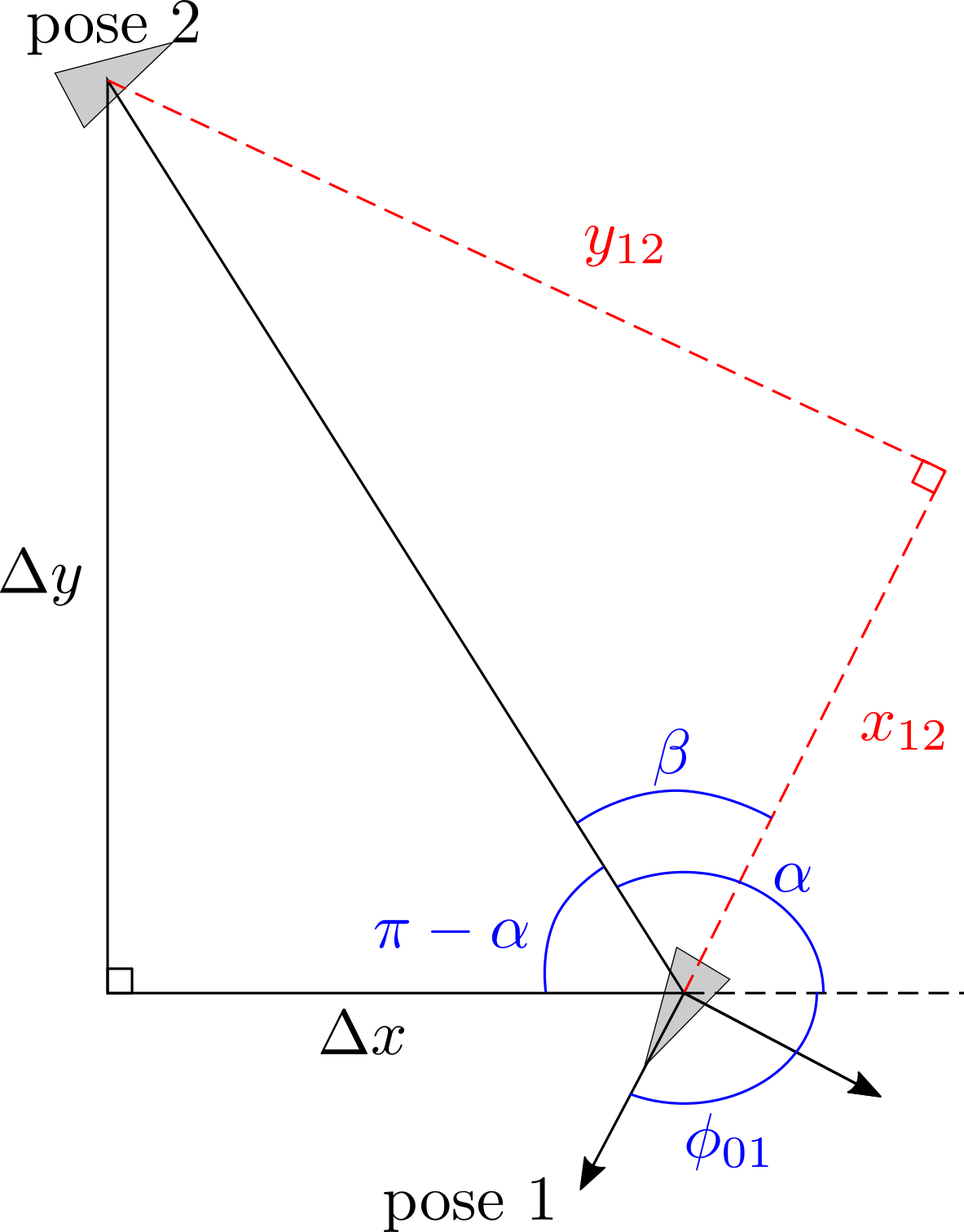}
	\caption{Drawing of geometry for the proof of Proposition \ref{prop:theta0=phi01}.}
	\label{fig:prooffigure}
\end{figure}
\end{shownto}
\end{proof}

\begin{proposition}\label{prop:f''>0}
	On $\phi_1\in(\phi_{01}-\frac{\pi}{2},\phi_{01} + \frac{\pi}{2})$, $f''_{1,k}(\phi_1)>0$.
\end{proposition}
\begin{proof}
	On $\phi_1\in(\phi_{01}-\frac{\pi}{2},\phi_{01} + \frac{\pi}{2})$, via Proposition \ref{prop:theta0=phi01}, $2a_0\cos(\phi_1-\phi_{01}) > 0$. Also, by direct calculation, 
	\begin{equation}\label{eqn:Q2}
	Q_2 = \frac{1}{\sigma^2}\begin{bmatrix}
	1 & 0 & 0\\
	0 & \frac{1}{2} & -\frac{1}{2}\\
	0 & -\frac{1}{2} & \frac{1}{2}
	\end{bmatrix} \geq 0.
	\end{equation} 
	Hence $f''_{1,k}(\phi_1) > 0$ for $\phi_1\in(\phi_{01}-\frac{\pi}{2},\phi_{01} + \frac{\pi}{2})$, $k\in\{-1,0,1\}$.
\end{proof}
\begin{proof} [Theorem \ref{thm:localMinimaOnRk}]
	Consider first $k = 1$, which corresponds to the top left triangle, where $\phi_1\in[\phi_{01}-\pi,\phi_{01}]$. Then, by \eqref{eqn:measurementsArePerfect},
	\begin{align}
	f_{1,1}'(\phi_1) &= 2a_0\sin(\phi_1-\theta_0)-\frac{2}{\sigma^2}(\phi_{01}-\phi_1) \nonumber\\
	&\quad + \frac{1}{\sigma^2}(\phi_{12}-\phi_{02} + 2\pi + \phi_1),\text{ and} \label{eqn:f'}\\
	f_{1,1}''(\phi_1) &= 2a_0\cos(\phi_1-\phi_{01}) + \frac{3}{\sigma^2}.
	\end{align} via \eqref{eqn:Q2}. We aim to show that $f_{1,1}'(\phi_1) = 0$ on an interval where $f_{1,1}''(\phi_1)>0$. Now, $f_{1,1}''(\phi_1) > 0$, if and only if 
	\begin{equation}\label{eqn:conditionForPositivef''}
	\cos(\phi_1-\phi_{01}) > -\frac{3}{2a_0\sigma^2}.
	\end{equation} We consider two cases: case ``a'', where $\frac{3}{2a_0\sigma^2} > 1$, and case ``b'', where $\frac{3}{2a_0\sigma^2}\in(0,1)$. Since we are not considering the non-differentiable boundaries of $\mathcal R_k$, we ignore the case where $\cos(\phi_1-\phi_{01}) = -3/2a_0\sigma^2$.
	
	In case ``a'', \eqref{eqn:conditionForPositivef''} is always true, so $f_{1,1}''(\phi_1) > 0$ for all $\phi_1\in\R$, and any critical point will be a minimum. We cannot solve \eqref{eqn:f'} directly, but we can show a solution exists. By substitution, it can be shown that $f_{1,1}'(\phi_{01}) > 0$, and that $f_{1,1}'(\phi_{01} - \pi) <0$. Since $f_{1,1}'(\phi_1)$ is continuous, by the intermediate value theorem, $f_{1,1}'(\phi_1) = 0$ for some $\phi_1\in(-\phi_{01}-\pi,\phi_{01})$. 
	
	For case ``b'', $f_{1,1}''(\phi_1) > 0$ on $(\phi_{01}-\frac{\pi}{2},\phi_{01}+\frac{\pi}{2})$. As in case ``a'', we still use the intermediate value theorem, but instead evaluate $f_{1,1}'(\phi_{01}-\frac{\pi}{2})$, which equals:
	\begin{equation}
	f_{1,1}'(\phi_{01}-\frac{\pi}{2}) = - 2a_0 + \frac{\pi}{2\sigma^2},
	\end{equation}
	which is strictly negative if $\frac{3}{2a_0\sigma^2}\in(0,1)$. Hence there exists a root of $f_{1,1}'(\phi_1)$ on $(\phi_{01}-\frac{\pi}{2},\phi_{01})$, where $f_{1,1}''(\phi_1) > 0$, which implies a minimum of $f_{1,1}(\phi_1)$ on $(\phi_{01}-\frac{\pi}{2},\phi_{01})$.
	
	It is trivial to check that for $\phi_1\in[\phi_{01}-\pi,\phi_{01}]$, $(\phi_1,\phi_{2,k=1}^\star(\phi_1))$ satisfies the inequality constraints at the boundaries of $\mathcal R_1$; hence the minimum is in $\mathcal R_1$.
	Since $\phi_1^\star$ minimizes $f_{1,k}(\phi_1)$, by Lemma \ref{lem:minOfLowDimProblemIsMinOfHighDim}, $\Phi^\star = (\phi_1^\star,\phi_{2,k=1}^\star(\phi_1^\star))$ is a global minimum of $f_{2,k=1}(\Phi)$. However, while $\Phi^\star$ is a global minimum of $f_k(\Phi)$, $f_k(\Phi) = f(\Phi)$ only on $\mathcal R_k$. By Remark \ref{rem:minfkDoesntMeanMinf} and Lemma \ref{lem:noGlobalMinimaInR_pm1}, $\Phi^\star$ is merely a local minimum of $f(\Phi)$. Again by Lemma \ref{lem:minOfLowDimProblemIsMinOfHighDim}, the minimum $\Phi^\star$ of $f_k(\Phi)$ and the existence of $P^\star(\Phi)$ implies that $(P^\star(\Phi^\star),\Phi^\star)$ is a global minimum of $F_k(P,\Phi)$, and by the same logic as above, a local minimum of $F(P,\Phi)$.
	

	The proof for $\mathcal R_{-1}$ uses the same steps and is omitted.
\end{proof}

\begin{proof} [Theorem \ref{thm:minIsUniqueForChordal}] 
 The critical points of $g$ are $\Phi$ such that $J(\Phi) = 0$. This implies $J_2(\Phi)$, the second element $J(\Phi)$, is zero, which yields:
	\begin{equation}
	\phi_1^{\star\star}(\phi_2) := 2\phi_2 - \phi_{02} - \phi_{12} + 2n_1\pi, \quad n_1\in\mathbb Z
	\end{equation} which makes $J_2(\Phi) = 0$ for all $\phi_2\in\R$. 
	\begin{shownto}{arxiv}
		Then, substituting this into $J_1(\Phi) = 0$, and letting $\eta(\phi_2) = \phi_{02} - \phi_2$,
		\begin{align*}
		J_1(\Phi) &= 0\\ 0 &= -b_0\sin(2\eta) + \sin(-\eta)\\
		2b_0\sin(\eta)\cos(\eta)&=-\sin(\eta). 
		\end{align*} Solving yields:
	\end{shownto}
	\begin{shownto}{ral}
		Then, substituting this into $J_1(\Phi) = 0$ and solving yields:
	\end{shownto}
	\begin{equation}
	\phi_2^{\star\star} =\begin{cases}
	\phi_{02} + n_2\pi, n_2\in\mathbb Z & \text{for } \sin(\phi_{02}-\phi_2) = 0\\
	\phi_{02} \pm \arccos\left(-\frac{1}{2b_0}\right) & \text{for }  \sin(\phi_{02}-\phi_2) \neq 0,
	\end{cases}
	\end{equation} where $b_0 = a_0 + 1$. Since $a_0>0$, $-\frac{1}{2b_0}$ is always in the interval $[-1,0]$, and hence a valid argument to arccos(). Hence choosing $\Phi = \Phi^{\star\star} =  [\phi_1^{\star\star}(\phi_2^{\star\star}),\phi_2^{\star\star}]$ yields $J(\Phi) = 0$.
	
	Hence on the closed region $\mathcal S$, there are eleven critical points: eight on the boundary of $\mathcal S$, and three on the interior. The eight critical points on the boundary are $\Phi = [\phi_{01};\phi_{02}] + \Phi_c$, where $\Phi_c\in \{[\pm\pi;\pm\pi],[\pm\pi;0],[0;\pm\pi]\}$; these are a subset of the case when $\sin(\phi_{02}-\phi_2) = 0$. At these critical points, $H$ is sign-indefinite,\begin{shownto}{arxiv} as the bottom right term is zero and the off-diagonal terms are nonzero.\end{shownto}
	
	In the case that $\sin(\phi_{02}-\phi_2) \neq 0$, there are two $\Phi^{\star\star}$ in $\mathcal S$. By evaluating $H(\Phi^{\star\star})$ and taking the Schur complement, it is easy to show that $H(\Phi^{\star\star}) < 0$. Hence two critical points associated with $\sin(\phi_{02} - \phi_2) \neq 0$ are maxima.
	Finally, we have $\Phi = [\phi_{01},\phi_{02}]$. Evaluating the Hessian here gives 
	\begin{equation}
	H = 2\begin{bmatrix}
	a_0+2 & -1 \\ -1 & 2
	\end{bmatrix},
	\end{equation} which is positive definite for $a_0 > 0$. Hence $\Phi=[\phi_{01},\phi_{02}]$ is a \textit{unique} minimum of $g(\Phi)$ on $\mathcal S$. By Lemma \ref{lem:minOfLowDimProblemIsMinOfHighDim} and uniqueness of $P^\star$, $[P^\star;\phi_{01};\phi_{02}]$ is a unique minimum of $G(P,\Phi)$ on $\R^4\times \mathcal S$.
\end{proof}

\begin{proof}[Theorem \ref{thm:noisyCaseChordal}] 
	We show by direct calculation that for $\varepsilon = \pi$, there are two distinct global minima in $\mathcal S$. We use $\phi^{\star\star}(\phi_2)$ as in the proof for Theorem \ref{thm:minIsUniqueForChordal} to make $J_2(\Phi) = 0$, but $\phi_2^{\star\star}$ is different. As in the proof of Theorem \ref{thm:minIsUniqueForChordal}, $b_0 = a_0 + 1$, and $\eta = \phi_{02}-\phi_2$. By \eqref{eqn:measurementsAreImperfect} with $\varepsilon = \pi$, 
	\begin{equation}
		J_1(\phi_1^{\star\star}(\phi_2),\phi_2) = -b_0\sin(2\gamma + \varepsilon) + \sin(-\gamma).
	\end{equation} We are looking for the critical points on the interior of $\mathcal S$, this time with $\varepsilon = \pi$. $J_1(\phi_1^{\star\star}(\phi_2),\phi_2) = 0$ can be rewritten:
	\begin{shownto}{arxiv}
	\begin{align*}
	J_1(\phi_1^{\star\star}(\phi_2),\phi_2) &= 0 \\
	&=-b_0\sin(2\eta + \varepsilon) + \sin(-\eta) \\
	&=b_0\sin(2\eta) - \sin(\eta) \\
	&=2b_0\sin(\eta)\cos(\eta)- \sin(\eta) \\
	&=\sin(\eta)(2b_0\cos(\eta)-1)
	\end{align*}
	\end{shownto}
	\begin{shownto}{ral}
		\begin{equation}
		0 = \sin(\eta)(2b_0\cos(\eta)-1)
		\end{equation}
	\end{shownto}
	For $\eta \in [-\pi,\pi]$, this has solutions:
	\begin{equation}
	\eta = n_3\pi \text{ and } \pm \arccos\left(1/{2b_0}\right)
	\end{equation} with $n_3\in\{-1,0,1\}$. These $\eta(\phi_2)$ result in critical points of $G(\Phi)$ (i.e. $J(\Phi) = 0$). 
We can rewrite the Hessian in terms of $\eta$:
	\begin{equation}
	H(\eta) = \begin{bmatrix}
	b_0\cos(2\eta + \varepsilon) + \cos(\eta) & -\cos(\eta) \\
	-\cos(\eta) & 2\cos(\eta)
	\end{bmatrix}.
	\end{equation} 
\begin{shownto}{arxiv}
	Then, at $\eta = \pm\arccos(\frac{1}{2b_0})$,
	\begin{equation}
	H\left(\pm\arccos\left(\frac{1}{2b_0}\right)\right)=	\begin{bmatrix}
	b_0 & -\frac{1}{2b_0} \\ -\frac{1}{2b_0} & \frac{1}{b_0}
	\end{bmatrix},
	\end{equation} which is positive definite for $b_0 > 1$.
\end{shownto}
\begin{shownto}{ral}
	Evaluating this at $\eta = \pm\arccos(1/2b_0)$ yields $H(\eta) > 0$ if $b_0 > 1$.
	\end{shownto} Hence two distinct critical points $\Phi_+$ and $\Phi_-$ exist, each of which are minima. 
	\begin{shownto}{arxiv}
	We check the other possible values of $\eta$ that yield critical points and show there cannot be minima associated with them. In the case $\eta = 0$, 
	\begin{equation}
	H(\eta = 0) = \begin{bmatrix}
	1 - b_0 & -1 \\ -1 & 2
	\end{bmatrix},
	\end{equation} which has $\det(H) = 1-2b_0 < 0$, so there cannot be a minimum where $\eta = 0$. Similarly, in the final possible case $\eta = \pm \pi$, 
	\begin{equation}
	H(\eta = \pm \pi) = \begin{bmatrix}
	-b_0 -1 & 1 \\ 1 & 2
	\end{bmatrix},
	\end{equation} which has trace$(H) = -b_0 -3 < 0$, so $H(\pm\pi)$ cannot be positive definite, so there cannot be a minimum at $\eta = \pm\pi$.
	\end{shownto}
	\begin{shownto}{ral}
		Checking $H(\eta)$ at $0,\pm\pi$ result in non-positive-definite $H(\eta)$. 
	\end{shownto} 
	~Hence the only minima are at $\eta = \pm\arccos(1/2b_0)$.
	Substituting $\Phi_\pm$ into $g(\Phi)$ results in the same cost, so both are global minima of $g(\Phi)$.
	\end{proof}

%
%


	\bibliography{library}
	\bibliographystyle{IEEEtran}

\end{document}